\documentclass[12pt]{amsart}
\usepackage{amsmath,amssymb,amscd,graphicx,verbatim}
\pdfoutput=1
\usepackage{graphicx}

\theoremstyle{definition}

\newtheorem{theorem}{Theorem}

\newtheorem{lemma}{Lemma}
\newtheorem*{remark}{Remark}
\newtheorem*{examples}{Examples}

\begin{document}
\title[Lower bounds for the principal genus]{Lower bounds for the principal genus\\ of definite binary quadratic forms}
\author{Kimberly Hopkins}
\address{Department of Mathematics\\University of Texas at Austin\\Austin, TX 78712-0257}
\email{khopkins@math.utexas.edu}
\author{Jeffrey Stopple}
\address{Department of Mathematics\\University of California, Santa Barbara\\Santa Barbara, CA 93106-3080}
\email{stopple@math.ucsb.edu}
\begin{abstract}
We apply Tatuzawa's version of Siegel's theorem to derive two lower bounds on the size of the principal genus of positive definite binary quadratic forms.
%The first bound is better than the second on certain intervals of discriminants, but the second bound is better as the discriminant goes to infinity.
\end{abstract}

\maketitle

\subsection*{Introduction}  Suppose $-D<0$ is a fundamental discriminant.  By genus theory we have an exact sequence for the class group $\mathcal C(-D)$ of positive definite binary quadratic forms:
\[
 \mathcal P(-D)\overset{\text{def.}}=\mathcal C(-D)^2 \hookrightarrow \mathcal C(-D) \twoheadrightarrow \mathcal C(-D)/ \mathcal C(-D)^2 \simeq (\mathbb Z/2)^{g-1},
\]
where $D$ is divisible by $g$ primary discriminants (i.e., $D$ has $g$ distinct prime factors).
Let $p(-D)$ denote the cardinality of the principal genus $\mathcal P(-D)$.  The genera of forms are the cosets of $\mathcal C(-D)$ modulo the principal genus, and thus $p(-D)$ is the number of classes of forms in each genus.  The study of this invariant of the class group is as old as the study of the class number $h(-D)$ itself.  Indeed, Gauss wrote in \cite[Art. 303]{Gauss}
\begin{quote}. . . Further, the series of [discriminants] corresponding to the same given classification (i.e. the given number of both genera and classes) always seems to terminate with a finite number . . . However, \emph{rigorous} proofs of these observations seem to be very difficult.
\end{quote}

Theorems about $h(-D)$ have usually been closely followed with an analogous result for $p(-D)$.  When Heilbronn \cite{He} showed that $h(-D)\to\infty$ as $D\to\infty$, Chowla \cite{C} showed that $p(-D)\to\infty$ as $D\to\infty$.
%Chowla's result appeared in the same journal issue as Heilbronn's, and his enthusiasm to appear in print detracted from the exposition - the crucial estimate appears in a footnote on the last page without proof.
An elegant proof of Chowla's theorem is given by Narkiewicz in \cite[Prop 8.8 p. 458]{N}.

Similarly, the Heilbronn-Linfoot result \cite{HL} that $h(-D)>1$ if $D>163$, with at most one possible exception was matched by Weinberger's result \cite{W} that $p(-D)>1$ if $D>5460$ with at most one possible exception.  On the other hand, Oesterl\'{e}'s \cite{Oes} exposition of the Goldfeld-Gross-Zagier bound for $h(-D)$ already contains the observation that the result was not strong enough to give any information about $p(-D)$.

In \cite{T} Tatuzawa proved a version of Siegel's theorem: for every $\varepsilon$ there is an explicit constant $C(\varepsilon)$ so that
\[
h(-D)>C(\varepsilon)D^{1/2-\varepsilon}
\]
with at most one exceptional discriminant $-D$.  This result has never been adapted to the study of the principal genus.  It is easily done; the proofs are not difficult so it is worthwhile filling this gap in the literature.   We present two versions.  The first version contains a transcendental function (the Lambert $W$ function discussed below).  The second version gives, for each $n\ge4$, a bound which involves only elementary functions. For each fixed $n$ the second version is stronger on an interval $I=I(n)$ of $D$ , but the first is stronger as $D\to\infty$. The second version has the added advantage that it is easily computable. 
  (N.B. The constants in Tatuzawa's result have been improved in \cite{Ho} and \cite{JL}; these could be applied at the expense of slightly more complicated statements.)

\subsection*{Notation}
We will always assume that $g\ge 2$, for if $g=1$ then $-D=-4,-8$, or $-q$ with $q\equiv 3\bmod 4$ a prime. In this last case $p(-q)=h(-q)$ and Tatuzawa's theorem \cite{T} applies directly.

\section*{First version}
\begin{lemma}\label{Lemma1}  If $g\ge 2$,
\[
\log(D)>g\log(g).
\]
\end{lemma}
\begin{proof}  Factor $D$ as $q_1,\ldots q_g$ where the $q_i$ are (absolute values) of primary discriminants, i.e. $4$, $8$, or odd primes.  Let $p_i$ denote the $i$th prime number, so we have
\begin{equation}\label{Eq:inequality}
\log(D)=\sum_{i=1}^g \log(q_i)\ge \sum_{i=1}^g \log(p_i)\overset{\text{def.}}=\theta(p_g).
\end{equation}
By \cite[(3.16) and (3.11)]{RS}, we know that Chebyshev's function $\theta$ satisfies $\theta(x)>x(1-1/\log(x))$ if $x>41$, and that
\[
p_g>g(\log(g)+\log(\log(g))-3/2).
\]
After substituting $x=p_g$ and a little calculation, this gives $\theta(p_g)>g\log(g)$ as long as $p_g>41$, i.e. $g>13$.  For $g=2,\ldots, 13$, one can easily verify the inequality directly.
\end{proof}

Let $W(x)$ denote the Lambert $W$-function, that is, the inverse function of $f(w)=w\exp(w)$ (see  \cite{E}, \cite[p. 146 and p. 348, ex 209]{PS}).  For $x\ge0$ it is positive, increasing, and concave down.  The Lambert $W$-function is also sometimes called the product log, and is implemented as \texttt{ProductLog} in \emph{Mathematica}.

\begin{theorem}\label{Theorem2}  If $0<\varepsilon<1/2$ and $D>\max(\exp(1/\varepsilon),\exp(11.2))$, then with at most one exception
\[
p(-D)>\frac{1.31}{\pi}\varepsilon D^{1/2-\varepsilon-\log(2)/W(\log(D))}.
\]
\end{theorem}
\begin{proof}
Tatuzawa's theorem \cite{T}, says that with at most one exception
\begin{equation}\label{Eq:tatuzawa}
\frac{\pi\cdot h(-D)}{ \sqrt{D}}=L(1,\chi_{-D})>.655\varepsilon D^{-\varepsilon},
\end{equation}
thus
\[
p(-D)=\frac{2h(-D)}{2^g}>\frac{1.31\varepsilon\cdot D^{1/2-\varepsilon}}{\pi \cdot 2^g}.
\]

The relation $\log(D)>g\log(g)$ is equivalent to
\begin{gather*}
\log(D)>\exp(\log(g))\log(g),\\
\intertext{Thus applying the increasing function $W$ gives, by definition of $W$}
W(\log(D))>\log(g),\\
\intertext{and applying the exponential gives}
\exp(W(\log(D))>g.
\end{gather*}
The left hand side above is equal to $\log(D)/W(\log(D))$ by the definition of $W$.  Thus
\begin{gather*}
-\log(D)/W(\log(D)) < -g,\\
D^{-\log(2)/W(\log(D))}=2^{-\log(D)/W(\log(D))} < 2^{-g},
\end{gather*}
and the Theorem follows.
\end{proof}
\begin{remark}

 Our estimate arises from the bound $\log(D)>g\log(g)$, which is nearly optimal.  That is, for every $g$, there exists a fundamental discriminant (although not necessarily negative) of the form
\[
D_g\overset{\text{def.}}=\pm 3\cdot4\cdot5\cdot7\dots p_g,
\]
and
\[
\log|D_g| = \theta(p_g)+\log(2).
\]
From the Prime Number Theorem we know $\theta(p_g)\sim p_g$, so
\[
\log|D_g| \sim p_g+\log(2)
\]
while \cite[3.13]{RS} shows $p_g<g(\log(g)+\log(\log(g))$ for $g\ge6$.
\end{remark}

\section*{Second version}

\begin{theorem}\label{Theorem1}
Let $n\ge 4$ be any natural number. If $0<\varepsilon<1/2$ and $D>\max(\exp(1/\varepsilon),\exp(11.2))$, then with at most one exception
\[
p(-D)>\frac{1.31\varepsilon}{\pi}\cdot \frac{D^{1/2-\varepsilon-1/n}}{f(n)},
\]
where
\[f(n) = \exp\big[ ( \pi(2^n) - 1/n ) \log 2 - \theta(2^n)/n\big];
\]
here $\pi$ is the prime counting function and   $\theta$ is the Chebyshev function.

\end{theorem}
\begin{proof}
First observe
\[
f(n) = \frac{2^{\pi(2^n)} }{2^{1/n}\prod_{\text{primes }p<2^n}p^{1/n}}.
\]
 
From Tatuzawa's Theorem (\ref{Eq:tatuzawa}), it suffices to show $2^g \leq f(n)D^{1/n}$.   Suppose first that $D$ is not $\equiv0\pmod8$.

Let $S = \{ 4, \: \text{odd primes} < 2^n\}$, so $|S| = \pi(2^n)$.  Factor $D$ as $q_1\cdots q_g$ where $q_i$ are (absolute values) of coprime primary discriminants, that is, $4$ or odd primes, and satisfy $q_i<q_j$ for $i<j$. Then, for some $0\leq m\leq g$, we have $q_1,\dots, q_m \in S$ and $q_{m+1},\dots, q_g\not\in S$, and thus $2^n<q_i$ for $i=m+1,\dots, g$. This implies
\begin{align*}
    2^{gn} &= \underbrace{2^n\cdots 2^n}_{m} \cdot \underbrace{2^n\cdots 2^n}_{g-m}
            \le 2^{mn}\  q_{m+1} q_{m+2} \ldots  q_g \\
            &= \frac{ 2^{mn} }{ q_1\cdots q_m} D
            \leq \frac{2^{|S|\cdot n}}{ \prod_{q\in S} q}\cdot D \\
            \intertext{ as we have included in the denominator the remaining elements of $S$ (each of which is $\le2^n$).  The above is}
            &= \frac{2^{\pi(2^n) \cdot n}}{ 2 \prod_{\text{primes }p<2^n}p} \cdot D
            = f(n)^n\cdot  D.
\end{align*}
This proves the theorem when $D$ is not  $\equiv 0\bmod 8$.  In the remaining case, apply the above argument to $D^\prime=D/2$; so
\[
2^{gn}\le f(n)^nD^\prime <f(n)^n D.
\]
\end{proof}
\begin{examples}If $0<\varepsilon<1/2$ and $D>\max(\exp(1/\varepsilon),\exp(11.2))$, then with at most one exception, Theorem \ref{Theorem1} implies
\begin{gather*}
p(-D)>0.10199\cdot\varepsilon\cdot D^{1/4-\varepsilon}\quad (n=4)\\
p(-D)>0.0426\cdot\varepsilon\cdot D^{3/10-\varepsilon}\quad (n=5)\\
p(-D)>0.01249\cdot\varepsilon\cdot D^{1/3-\varepsilon}\quad (n=6)\\
p(-D)>0.00188\cdot\varepsilon\cdot D^{5/14-\varepsilon}\quad (n=7)\\
\end{gather*}
\end{examples}

\section*{Comparison of the two theorems} How do the two theorems compare?  Canceling the terms which are the same in both, we seek inequalities relating
\[
D^{-\log 2/W(\log D)} \quad\text{v.}\quad  \frac{D^{-1/n}}{f(n)}.
\]
\begin{theorem}\label{Theorem3}
For every $n$, there is a range of $D$ where the bound from Theorem \ref{Theorem1} is better than the bound from Theorem \ref{Theorem2}. However, for any fixed $n$ the bound from Theorem \ref{Theorem2} is eventually better as $D$ increases.
\end{theorem}
For fixed $n$, the first statement of Theorem \ref{Theorem3} is equivalent to proving
\[
D^{\log(2)/W(\log(D)) - 1/n}\ge f(n)
\]
on a non-empty compact interval of the $D$ axis.  Taking logarithms, it suffices to show,
\begin{lemma}\label{T:compare}
Let $n\ge 4$. Then
\[
        x\bigg( \frac{\log 2}{W(x)} -\frac{1}{n}\bigg) \geq \log f(n)
\]
on some non-empty compact interval of positive real numbers $x$.
\end{lemma}
\begin{proof}
Let $g(n,x) = x\, ( \log 2/W(x) - 1/n )$. Then
\[
\frac{\partial g}{\partial x} = \frac{\log 2}{W(x)+1} - \frac1n \qquad \text{and}\qquad \frac{\partial^2 g}{\partial x^2} = \frac{-\log 2 \cdot W(x)}{ x( W(x)+1)^3}.
\]
This shows $g$ is concave down on the positive real numbers and has a maximum at
\[
x = 2^n (n\log 2-1)/e.
\]
Because of the concavity, all we need to do is show that $g(n,x)>\log f(n)$ at \emph{some} $x$.  The maximum point is slightly ugly so instead we
let $x_0 = 2^n n\log 2/e$.
%Then it suffices to show $g(n,x_0)\geq f(n)$ for all $n>4$.
%(?? check lower bound on $n$!!)

Using $W(x) \sim \log x - \log\log x$, a short calculation shows
\[
        g(n,x_0) \sim \frac{1}{e}\cdot\frac{2^n}{n}.
\]

  By \cite[5.7)]{RS2}, a lower bound on Chebyshev's function is
\[
   \theta(t)> t\bigg( 1- \frac{1}{40 \log t}\bigg) , \quad t>678407.
\]
(Since we will take $t=2^n$ this requires $n>19$ which is not much of a restriction.)\ \
By \cite[(3.4)]{RS}, an upper bound on the prime counting function is
\[
\pi(t) < \frac{t}{\log t -3/2}, \quad t>e^{3/2}.
\]

Hence $-\theta(2^n)< 2^n\, (1/(40 n\log 2) -1)$ and so
\begin{align*}
\log f(n) &= \bigg(\pi(2^n)-\frac1n\bigg)\log 2 - \frac{\theta(2^n)}{n} \\
            &< \bigg( \frac{2^n}{n\log 2 - 3/2} -\frac1n\bigg)\log2 + \frac{2^n}{n} \bigg(\frac{1}{40 n\log2} -1\bigg)\\
            &\sim  \frac{61}{40\log2}\cdot\frac{2^n}{n^2}.
\end{align*}

Comparing the two asymptotic bounds for $g$ and $\log f$ respectively we see that
\[
 \frac{1}{e}\cdot\frac{2^n}{n} >   \frac{61}{40\log2}\cdot\frac{2^n}{n^2},
\]
 for $n\ge 6$; small $n$ are treated by direct computation.\footnote{The details of the asymptotics have been omitted for conciseness.}
\end{proof}

\begin{figure}
\begin{center}
\includegraphics[scale=1, viewport=0 0 400 200,clip]{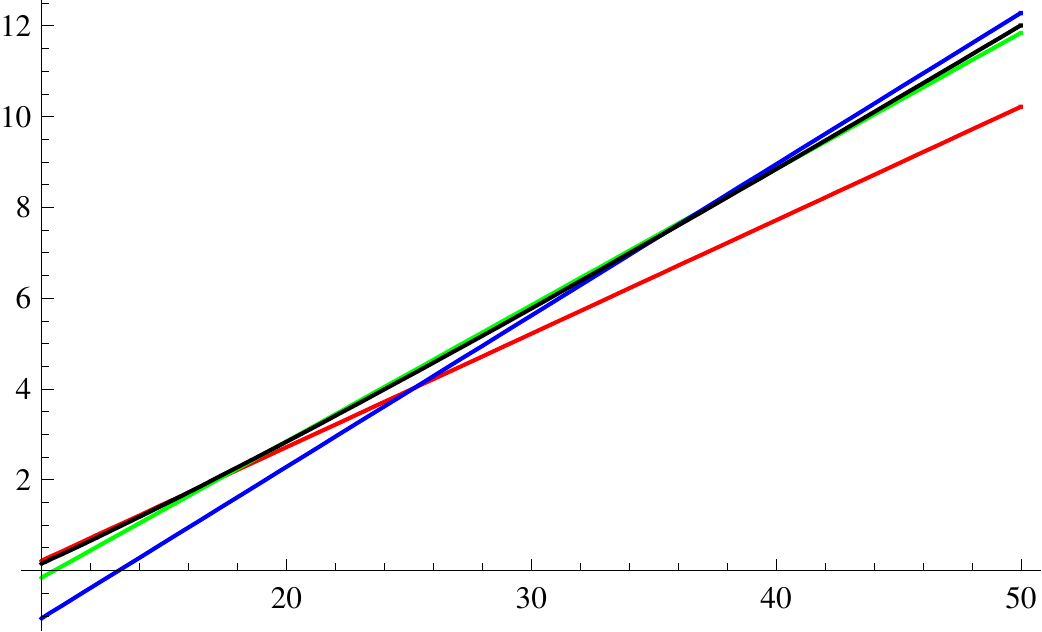}
\caption{$\log$-$\log$ plots of the bounds from Theorems \ref{Theorem2} and \ref{Theorem1}}\label{F:pic}
\end{center}
\end{figure}

Figure \ref{F:pic} shows a $\log$-$\log$ plot of the two lower bounds, omitting the contribution of the constants which are the same in both and the terms involving $\varepsilon$.  That is, Theorem \ref{Theorem1} gives  for each $n$ a lower bound  $b(D)$ of the form
\begin{gather*}
b(D)=C(n) \varepsilon D^{1/2-1/n-\varepsilon},\quad\text{so}\\
\log(b(D)) = (1/2-1/n-\varepsilon)\log(D)+\log(C(n))+\log(\varepsilon).
\end{gather*}
Observe that for fixed $n$ and $\varepsilon$, this is linear in $\log(D)$, with the slope an increasing function of the parameter $n$.
What is plotted is actually   $(1/2-1/n)\log(D)+\log(C(n))$ as a function of $\log(D)$, and analogously for Theorem \ref{Theorem2}.\ \
In red, green, and blue are plotted the lower bounds from Theorem \ref{Theorem1} for $n=4$, $5$, and $6$ respectively.    In black is plotted the lower bound from Theorem \ref{Theorem2}.

\begin{examples}The choice $\varepsilon=1/\log(5.6\cdot 10^{10})$ in Theorem \ref{Theorem2} shows that  $p(-D)>1$ for $D>5.6\cdot 10^{10}$ with at most one exception.  (For comparison, Weinberger \cite[Lemma 4]{W} needed $D>2\cdot 10^{11}$ to get this lower bound.)
And, $\varepsilon=1/\log(3.5\cdot 10^{14})$ in Theorem \ref{Theorem2} gives $p(-D)>10$ for $D>3.5\cdot 10^{14}$ with at most one exception.  Finally,  $n=6$ and $\varepsilon=1/\log(4.8\cdot 10^{17})$ in Theorem \ref{Theorem1} gives $p(-D)>100$ for $D>4.8\cdot 10^{17}$ with at most one exception.
\end{examples}


\begin{thebibliography}{99}
\bibitem{C} S. Chowla, \emph{An extension of Heilbronn's class-number theorem}, Quarterly J. Math., \textbf{5} (1934), pp. 150-160.
\bibitem{E} L. Euler, \emph{De serie Lambertiana plurimisque eius insignibus proprietatibus}, Opera Omnia Ser. 1 Vol. 6, pp. 350-369.
\bibitem{Gauss} C. F. Gauss, Disquisitiones Arithmeticae, Yale Univ. Press, 1966.
\bibitem{He} H. Heilbronn, \emph{On the class-number in imaginary quadratic fields}, Quarterly J. Math., \textbf{5} (1934), pp. 304-307.
\bibitem{HL} H. Heilbronn and E. Linfoot, \emph{On the imaginary quadratic corpora of class-number one}, Quarterly J. Math., \textbf{5} (1934), pp. 293-301.
\bibitem{Ho} J. Hoffstein, \emph{On the Siegel-Tatuzawa theorem}, Acta Arith. \textbf{XXXVIII} (1980), pp. 167-174.
%\bibitem{I} H. Iwaniec, \emph{Conversations on the exceptional character}, Springer LNM v. 1891.
\bibitem{JL}  C.G. Ji and H.W. Lu, \emph{Lower bound of real primitive $L$-function at $s=1$},
Acta Arith. \textbf{111} (2004) no. 4, pp. 405-409.
\bibitem{N} W. Narkiewicz, Elementary and Analytic Theory of Algebraic Numbers, 2nd. ed. Springer-Verlag, 1990.
\bibitem{Oes} J. Oesterl\'{e}, \emph{Nombres de classes des corps quadratiques imaginaires}, S\'{e}m. Bourbaki, vol. 1983/84,  Ast\'{e}risque, no. 121-122 (1985), pp. 309-323.
\bibitem{PS} G. P\'{o}lya and G. Szeg\"{o}, Aufgaben und LehrsŠtze der Analysis. Berlin, 1925. Reprinted as Problems and Theorems in Analysis I. Berlin: Springer-Verlag, 1998.
\bibitem{RS}  J.B. Rosser and L. Schoenfeld, \emph{Approximate formulas for some functions of prime numbers}, Illinois J. Math., \textbf{6} (1962), pp. 64-94.
\bibitem{RS2} \bysame, \emph{Sharper bounds for the Chebyshev functions $\theta(x)$ and $\psi(x)$}, Math. Comp., \textbf{29} (1975), pp. 243-369.
\bibitem{T}  T. Tatuzawa, \emph{On a theorem of Siegel}, Jap. J. Math. \textbf{21} (1951), pp. 163-178.
\bibitem{W} P. Weinberger, \emph{Exponents of the class groups of complex quadratic fields}, Acta Arith. \textbf{XXII}, 1973, pp. 117-124.

\end{thebibliography}
\end{document}